\documentclass{article}
\usepackage{amssymb}
\usepackage{amsmath}
\usepackage{amscd}
\usepackage{amsthm}
\usepackage{caption}
\usepackage{pgf}
\usepackage{comment}
\usepackage{mathtools}
\usepackage{enumerate}
\newtheorem{theorem}{Theorem}
\newtheorem*{theorem*}{Theorem}
\newtheorem*{lemma*}{Lemma}
\newtheorem*{corollary*}{Corollary}

\newtheorem{proposition}[theorem]{Proposition}
\newtheorem{definition}[theorem]{Definition}
\newtheorem{lemma}[theorem]{Lemma}

\newtheorem{corollary}[theorem]{Corollary}
\newcommand{\eend}{\mbox{End}}
\newcommand{\diff}{\mbox{Diff}}

\newcommand{\der}{\mbox{Der}}
\newcommand{\Hom}{\mbox{Hom}}
\newcommand{\hder}{\mbox{HDer}}

\newcommand{\aut}{\mbox{Aut}}

\title{Differential Operators on Modular Extensions}
\author{Matthew Wechter}

\begin{document}
\maketitle

\section{Introduction}
\indent Classical Galois theory is the study of finite separable field extensions and their relation to the group of relative automorphisms of a field extension.  If $P/k$ is a field extension, there exists a unique maximal intermediate subfield $\Sigma$ such that $\Sigma/k$ is separable \cite[IV, \S1]{Jacobson}.  $\Sigma$ is called the {\bfseries separable closure} of $k$ in $P$.  Classical Galois theory provides information about $\Sigma/k$ when $\Sigma/k$ is a finite extension, but no information about the purely inseparable extension $P/\Sigma$.

The lattice of subfields of a finite purely inseparable extension is often very large, and we will restrict our attention to a certain type of purely inseparable extension.  A finite field extension $L/K$ is called {\bfseries modular} if there exists a subset $\{x_1,\ldots, x_n\}$ such that $\displaystyle L\cong K(x_1)\otimes_K \cdots\otimes_K K(x_n)$.  The first careful study of modular extensions was begun by Sweedler \cite{Sweedler}, and Chase \cite{ChaseGalois} proved the existence of a Galois correspondence between modular extensions and special sub-group schemes of the automorphism scheme of a purely inseparable extension.  

We seek to describe modular extensions by studying the endomorphisms of the field extension.  It is a consequence of the Jacobson-Bourbaki Theorem \cite[p.22 Theorem 2]{Jacobson} that there is a $1-1$ correspondence between intermediate subfields of a finite purely inseparable extension $L/K$ and certain subalgebras of the ring of $K$-linear endomorphisms of the extension.  In fact, unlike the classical Galois theory regarding intermediate normal subextensions of a field extension, the Jacobson-Bourbaki Theorem  gives a Galois correspondence between intermediate subfields and subalgebras of the algebra of endomorphisms in which the type of subfield or field extension is irrelevant.

For any ring extension $B/A$, define $\diff_AB$ to be the ring of differential operators of $B$ which are linear with respect to $A$.  Combining the Jacobson-Bourbaki Theorem with the well-known fact that $\eend_KL=\diff_KL$ whenever $L/K$ is finite purely inseparable, one can ask whether the subalgebras $\diff_{K'}L\subset\diff_KL$ which correspond to intermediate subfields $L/K'/K$ with $L/K'$ modular can be characterized?

Setting $\mbox{char}\,K=p>0$, in this paper we define $L^{p^i}$-subspaces of the $L$-vector space $\diff_KL$, denoted by $\mathcal{A}_{L/K,\,i}\subset \diff_KL$, which are characterized by the action of the differential operators on the subfields $L^{p^i}\subset L$.  Using these objects, the question above is answered by the main theorem:

\newtheorem*{maintheorem}{Theorem \ref{thm:modulardiffops}}
\begin{maintheorem}
Let $L/K$ be a finite purely inseparable extension of exponent $e$ and let $\mbox{char } K=p>0$.  Then $L/K$ is modular if and only if for all $0<i\le e-1$, the multiplication homomorphism $L\otimes_{L^{p^i}}\mathcal{A}_{L/K,\,i}\to\diff^{\,p^i}_KL$ is a surjection, where $\diff^{\,p^i}_KL$ is the space of differential operators of $L/K$ of order $\le p^i$.  That is, for each $i$, $\mathcal{A}_{L/K,\,i}$ spans $\diff^{\,p^i}_KL$ as an $L$-subspace.
\end{maintheorem}

It follows directly from this theorem that we can use these subsets to determine the maximal modular subextension of a finite purely inseparable extension:

\newtheorem*{ccorollary}{Corollary \ref{cor:biggestmodularextension}}
\begin{ccorollary}
Let $L/K$ be as in the theorem.  Let $\mathcal{D}$ be the $L$-subalgebra of $\diff_KL$ generated by the $\mathcal{A}_i$.  Then $\mathcal{D}$ is the largest subalgebra of $\diff_KL$ such that $L/L^{\mathcal{D}}$ is a modular extension.
\end{ccorollary}

The first section of this paper will provide a simple introduction to differential operators on purely inseparable extensions, while the second section will state properties of generators for such extensions.  The third section contains the main theorem and corollary, as well as a technical lemma that may have future use to provide explicit constructions of differential operators on field extensions.

\section{Differential Operators on Finite Purely Inseparable Extensions}\label{ch:Jacobson}

\indent   To study the endomorphisms of a finite purely inseparable extension, it suffices to study the ring of differential operators of the extension.

\begin{proposition}\label{prop:enddiff}
Let $L/K$ be a finite purely inseparable extension and suppose $K$ has characteristic $p>0$.  Then $\diff_KL=\eend_KL$.
\end{proposition}

\begin{proof}
Since $L/K$ is a finite extension, the degree of the extension is a power of $p$.  Hence there exists an $e\in \mathbb{Z}$ such that $L^{p^e}\subseteq K$.  Thus, if $I$ is the ideal of the diagonal of $L/K$, then $\displaystyle I^{p^e}=0\in L\otimes_K L$ and $\displaystyle \left(L\otimes_K L\right)/I^{p^e}\cong L\otimes_K L$.  Since there is a $1-1$ correspondence between elements of $\eend_KL$ and $\Hom_L(L\otimes_K L,\,L)$, where $L$ acts on $L\otimes_K L$ on the left, then every endomorphism in $\eend_KL$ is a differential operator of order $\le p^e$ by the definition of a differential operator (cf. \cite[\S 16.8]{EGAIV}).
\end{proof}

Galois correspondences for purely inseparable extensions were first studied by Jacobson, and he showed a close relationship between exponent $1$ extensions and their modules of derivations.  This correspondence follows directly from the Jacobson-Bourbaki Theorem and cannot be extended to include extensions of higher exponent.  Sweedler's study of modular extensions relates Hasse-Schmidt derivations (or higher derivations) \cite{hasse} to certain modular subextensions and provides motivation for studying the higher order differential operators on purely inseparable extensions.

\begin{theorem}\label{Sweedlertheorem} \cite[Theorem 1]{Sweedler}
Let $k$ be a field of characteristic $p>0$ and suppose $K/k$ is a finite purely inseparable extension of exponent $e$.  The following are equivalent:
\begin{enumerate}
\item\label{Sweedler1}
$\displaystyle K/k$ is modular.
\item\label{Sweedler2}
There exist higher derivations of $K$ for which $k$ is the subfield of constants.
\item\label{Sweedler3}
$K^{p^i}$ is linearly disjoint from $k$ for all positive integers $i$.
\end{enumerate}
\end{theorem}

Note that for any finite purely inseparable extension $K/k$, if $F$ is the subfield of constants of all higher derivations of $K/k$, then $K/F$ is modular and $F$ will be the smallest intermediate subfield such that $K/F$ is modular.  By the Jacobson-Bourbaki Theorem, any intermediate subfield $k\subseteq F\subseteq K$ with $K/F$ modular corresponds to a unique $K$-subalgebra of $\eend_kK$.  The main theorem will give an intrinsic description of these subalgebras.

We end this section by defining special classes of differential operators that are required to distinguish modular from non-modular extensions.  Let $A$ be a ring of prime characteristic $p>0$ and $B$ a commutative $A$-algebra.  For any positive integer $n$, write $\diff_A^nB$ for the module of differential operators of $B/A$ of order $\le n$ \cite[\S16.8.1]{EGAIV}.  Define 

$$\mathcal{A}_{B/A,\,i}=\{D\in\diff^{\,p^i}_AB:\forall j\le i,\,D(B^{p^j})\subseteq B^{p^j}\}.$$

\noindent Note that $\mathcal{A}_{B/A,\,i}$ is a $B^{p^i}$-submodule of $\diff^{\,p^i}_AB$.  For example, let $L/K$ be a modular field extension with a $p$-basis $\{x_1,\ldots, x_n\}$ such that $e_i=\exp[x_i:K]>1$ for each $1\le i\le n$.  Then a $K$-basis of $L$ is $\displaystyle\left\{x_1^{j_1}\cdots x_n^{j_n}\right\}_{0\le j_i<p^{e_i}}$.  For $m<e_i$, let $\displaystyle \left(\frac{d}{dx_i}\right)^{[m]}$ be the unique endomorphism of $L/K$ sending $\displaystyle x_1^{j_1}\cdots x_n^{j_n}$ to $\displaystyle \binom{j_i}{m}x_1^{j_1}\cdots x_i^{j_i-m}\cdots x_n^{j_n}$.  

$\mathcal{A}_{L/K,\,1}$ consists of all differential operators of order $\le p$ which map $L^p\to L^p$.  Since $\displaystyle \left(\frac{d}{dx_i}\right)^{[i]}(L^p)=0$ for $i< p$, 
\begin{align*}
\mathcal{A}_{L/K,\,1}=L^p&\bigoplus_{0<i_1+\cdots i_n<p}L\left(\frac{d}{dx_1}\right)^{[i_1]}\cdots \left(\frac{d}{dx_n}\right)^{[i_n]}\\
&\bigoplus_{0<i\le n}L^p\left(\frac{d}{dx_i}\right)^{[p]}.
\end{align*}
Similarly,
\begin{align*} 
\mathcal{A}_{L/K,\,2}=L^{p^2}&\bigoplus_{0<i_1+\cdots i_n<p}L\left(\frac{d}{dx_1}\right)^{[i_1]}\cdots \left(\frac{d}{dx_n}\right)^{[i_n]}\\
&\bigoplus_{0<i_1+\cdots i_n<p} L^p\left(\frac{d}{dx_1}\right)^{[pi_1]}\cdots \left(\frac{d}{dx_n}\right)^{[pi_n]}\\
&\bigoplus_{0<i\le n}L^{p^2}\left(\frac{d}{dx_i}\right)^{[p^2]}.
\end{align*}


For an integer $i<\exp[L:K]$, write $\Gamma^{[i]}$ for the the $i$th divided powers functor.  Since $\der_KL$ is an $L$-vector space,  
$$\displaystyle \Gamma^{[p^i]}\left(\der_KL\right)\Big/\left(\sum_{0<j<i}\Gamma^{[p^j]}\left(\der_KL\right)\otimes_L\Gamma^{[p^{i-j}]}\left(\der_KL\right)\right)$$

\noindent is naturally an $L^{p^i}$-vector space.  Call this space the \emph{$p^i$th indecomposable divided powers} of $\der_KL$, and denote it by $\displaystyle\overline{\Gamma}^{[p^i]}\left(\der_KL\right)$.  There is a natural map $\displaystyle\overline{\gamma}^{p^i}:\,\der_KL\to\overline{\Gamma}^{[p^i]}\left(\der_KL\right)$ which sends an element $D\in\der_KL$ to the image of $\gamma^{p^i}(D)\in\Gamma^{[p^i]}\left(\der_KL\right)$ in $\displaystyle \overline{\Gamma}^{[p^i]}\left(\der_KL\right)$.  $\displaystyle\overline{\gamma}^{[p^i]}$ respects addition and induces an $L^{p^i}$-space isomorphism between $L^{p^i}\otimes_L\der_KL$ and $\displaystyle\overline{\Gamma}^{[p^i]}\left(\der_KL\right)$ where the $L$-action on $L^{p^i}$ is induced by the $i$th Frobenius map.  Finally, the image of $\displaystyle\mathcal{A}_{L/K,\,i}$ under the symbol map is isomorphic to both of these $L^{p^i}$-vector spaces.



\section{Pickert Generating Sequences}
To better study the differential operators on a purely inseparable field extension, it is worthwhile to determine more properties and invariants of the extension.  For a purely inseparable extension $L/K$ and any $x\in L$, set $\exp[x:K]$ to be the exponent of $x$ in $K$, and $\exp[L:K]$ to be the exponent of $L/K$.
\begin{definition}\label{def:Pickertgeneratingsequence}
Let $L/K$ be a finite purely inseparable extension of fields of characteristic $p>0$.  A sequence $\{x_1,\ldots,x_n\}\subset L$ is called a {\bfseries Pickert generating sequence} if the $x_i$ form a $p$-basis for $L/K$ and for each $i$, 
$$\exp\left[K(x_1,\ldots,x_i):K(x_1,\ldots, x_{i-1})\right]=\exp\left[x_i:K(x_1,\ldots,x_{i-1})\right].$$
\end{definition}

Any $p$-basis of a finite purely inseparable extension can be ordered to make it a Pickert generating sequence.  Let $e_i$ denote the exponents in Definition \ref{def:Pickertgeneratingsequence}.  For any $\alpha\in L$, $\exp[\alpha:K(x_1,\ldots,x_{i-1})]\le e_i$.  Hence $e_i\ge \exp[x_{i+1}:K(x_1,\ldots,x_{i-1})]$, and $\exp[x_{i+1}:K(x_1,\ldots,x_{i-1})]\ge\exp[x_{i+1}:K(x_1,\ldots,x_{i})]=e_{i+1}$.  Thus $e_1\ge e_2\ge\cdots\ge e_n$, and 
$$\{x_1^{r_1}x_2^{r_2}\cdots x_n^{r_n}\}_{0\le r_i<p^{e_i}}$$ 
\noindent is a $K$-basis for $L$.  Rasala \cite{Rasala} showed that the sequence of exponents is independent of the choice of $p$-basis for a finite extension.

Pickert \cite{Pickert} originally proved the following theorem, which Rasala applied in his work studying properties of purely inseparable extensions. 

\begin{proposition}\label{prop:structurequations}\cite[\S3, Theorem 1]{Rasala}
Let $L/K$ be a finite purely inseparable extension of fields of characteristic $p>0$.  Suppose $\{x_1,\ldots,x_n\}$ is a Pickert generating sequence for $L/K$ with corresponding exponent sequence $\{e_i\}$.  For each $i$, 
$$x_i^{p^{e_i}}\in K(x_1^{p^{e_i}},\ldots,x_{i-1}^{p^{e_i}}).$$
\end{proposition}
By the above proposition, for each $x_i$ in a Pickert generating sequence, there is a corresponding $g_i\in K(x_1^{q_i},\ldots,x_{i-1}^{q_i})$ such that $x_i^{q_i}=g_i$.  These polynomials in the $x_i$ will be called the {\bfseries structure equations} for $L/K$ corresponding to the Pickert generating sequence $\{x_1,\ldots,x_n\}$, and 

$$L\cong K[x_1,\ldots,x_n]\big/(x_1^{q_1}-g_1,\,x_2^{q_2}-g_2,\ldots,x_n^{q_n}-g_n)$$ 

\noindent as $K$-algebras.  For each $i$, the definition of $e_i$ guarantees that $q_i$ is the minimal power of $p$ for which the structure equations have the property described in Proposition \ref{prop:structurequations}.\\
\indent  Pickert generating sequences provide another criterion for determining modularity:

\begin{proposition}\label{prop:pickertmodular}\cite[\S5, Theorem 4]{Rasala}
Let $K$ be a field of characteristic $p>0$ and suppose $L/K$ is a finite purely inseparable field extension with Pickert generating sequence $\{x_1,\ldots,x_n\}$ and corresponding exponent sequence $e_1\ge\cdots\ge e_n$.  $L/K$ is modular if and only if for all $i$, $e_i=\exp[x_i:K]$.
\end{proposition}

Modularity can be determined from the exponent sequence by Proposition \ref{prop:pickertmodular}, but this determination will still be non-intrinsic since a choice of $p$-basis is required in the proposition.  The structure equations of an extension, however, can be related to behavior of the differential operators, allowing us to construct another test for modularity.


\section{Differential Operators on Modular Field Extensions}
The filtration of the differential operators by their orders provides the necessary information to determine which subrings of $\eend_KL$ correspond to modular extensions by the Jacobson-Bourbaki Theorem.  To prove the main theorem, the following technical lemma is required.

\begin{lemma}\label{lem:diffopextension}
Let $K$ be a field of characteristic $p>0$ and suppose $L/K$ is a finite purely inseparable extension of $K$.  Let $\{x_1,\,x_2,\ldots,x_n\}$ be a Pickert generating sequence for $L/K$ with corresponding exponent sequence $e_1\ge e_2\ge\cdots \ge e_n$ and let $D$ be a differential operator of order $N$ in $\diff_KK(x_1,\ldots,x_i)$ for some $i<n$.  Suppose $\widetilde{D}\in \diff_KK(x_1,\ldots,x_{i+1})$ is the unique extension of $D$ such that $\left.\widetilde{D}\right|_{K(x_1,\ldots,x_i)}=D$ and $\widetilde{D}(x_{i+1}^j)=0$ for all $0\le j<p^{e_{i+1}}$.  Then $\widetilde{D}$ is a differential operator of order $N$.
\end{lemma}

\begin{proof}
Set $q_i=p^{e_i}$ and let $f_{i+1}$ be the structure equation for $x_{i+1}$ with respect to the Pickert generating sequence $\{x_1,\ldots,x_n\}$.  Thus $x_{i+1}^{q_{i+1}}=f_{i+1}(x_1,\ldots,x_i)$ where, by Proposition \ref{prop:structurequations}, the degree of each $x_j$ in $f_{i+1}$ is a multiple of $q_{i+1}$.  Note that $\displaystyle K(x_1,\ldots, x_{i+1})=\bigoplus_{j=0}^{q_{i+1}-1}K(x_1,\ldots,x_i)x_{i+1}^j$.  Therefore, the map $\widetilde{D}$ as defined in the statement of the lemma is a well-defined endomorphism of $K(x_1,\ldots,x_{i+1})$ over $K$ and is unique by the direct sum decomposition of $K(x_1,\ldots, x_{i+1})$.  Since $K(x_1,\ldots,x_{i+1})/K$ is a finite purely inseparable extension, every endomorphism is a differential operator by Proposition \ref{prop:enddiff}.  Hence, $\widetilde{D}\in\diff_KK(x_1,\ldots,x_{i+1})$.\\
\indent The order of $\widetilde{D}$ is greater than or equal to the order of $D$ as differential operators in $\diff_KK(x_1,\ldots,x_i)$ and $\diff_KK(x_1,\ldots,x_{i+1})$, respectively.  Thus, it remains to show that the order of $\widetilde{D}$ is $N$.  Suppose $M$ is the order of $\widetilde{D}$.  Then there exist $a_1,\ldots, a_M$ such that 
\begin{equation}\label{eq:commutator1}
\left[\left[\cdots\left[\widetilde{D},\,a_1\right],\,a_2\right],\ldots,a_M\right]=z\in K(x_1,\ldots,x_{i+1})
\end{equation}
\noindent where $z\ne 0$.  This expression is symmetric in the $a_i$ and is a derivation in each commutator.  That is, for any $a,\,b\in K(x_1,\ldots,x_{i+1})$

\begin{eqnarray*}
\left[\left[\cdots\left[\widetilde{D},\,a_1\right],\,a_2\right],\ldots,ab\right]&=&a\left[\left[\cdots\left[\widetilde{D},\,a_1\right],\,a_2\right],\ldots,b\right]\\
&+&b\left[\left[\cdots\left[\widetilde{D},\,a_1\right],\,a_2\right],\ldots,a\right].
\end{eqnarray*}

Hence, Equation \ref{eq:commutator1} can be decomposed as a sum of commutators such that one of the summands is nonzero.  In particular, there exists a nonnegative integer $J$ and $g_1\ldots g_{M-J}\in K(x_1,x_2,\ldots,x_i)$ such that 

\begin{equation}\label{eq:commutatorunsimplified}
\left[\left[\left[\cdots\left[\left[\left[\left[\widetilde{D},\,x_{i+1}\right],\,x_{i+1}\right],\ldots,x_{i+1}\right],g_1\right],g_2\right]\ldots,\right],g_{M-J}\right]\ne 0
\end{equation}

\noindent where the number of $x_{i+1}$'s in this expression is $J$.  By Gerstenhaber \cite[Lemma 5.1]{GerstFundForm}, $q_{i+1}$ divides $J$.  If $J=cq_{i+1}$ where $c$ is a nonnegative integer, then using a formula proven by Nakai \cite[Corollary 11.2]{Nakai}, Equation \ref{eq:commutatorunsimplified} simplifies to

\begin{equation}\label{eq:commutator}
\left[\left[\left[\cdots\left[\left[\left[\left[\widetilde{D},\,f_{i+1}\right],\,f_{i+1}\right],\ldots,f_{i+1}\right],g_1\right],g_2\right]\ldots,\right],g_{M-J}\right]\in K(x_1,\ldots,x_{i+1})^\times
\end{equation}
where the number of $f_{i+1}$'s in the expression is $c$.  This equation is the commutator of $\widetilde{D}$ with $M-J+c$ elements of $K(x_1,\ldots,x_i)$.  Hence, Equation \ref{eq:commutator} equals

\begin{equation}\label{eq:commutatoroverK}
\left[\left[\left[\cdots\left[\left[\left[\left[D,\,f_{i+1}\right],\,f_{i+1}\right],\ldots,f_{i+1}\right],g_1\right],g_2\right]\ldots,\right],g_{M-J}\right]
\end{equation}

Each $f_{i+1}$ is a polynomial of degree at least $q_{i+1}$, hence the order of the differential operator $[D,\,f_{i+1}]$ is at most $N-q_{i+1}$.  Therefore the order of the differential operator in Equation \ref{eq:commutatoroverK} is at most $N-cq_{i+1}-(M-J)=N-M$.  This iterative commutator is a differential operator of order $0$ by its original construction, hence $N=M$ and the order of $\widetilde{D}$ equals the order of $D$.

\end{proof}

\begin{theorem}\label{thm:modulardiffops}
Let $L/K$ be a finite purely inseparable extension with exponent $e$ and let $\mbox{char }K=p>0$.  $L/K$ is modular if and only if for all $0<i\le e-1$, the multiplication homomorphism $L\otimes_{L^{p^i}}\mathcal{A}_{L/K,\,i}\to\diff^{\,p^i}_KL$ is a surjection.  That is, for each $i$, $\mathcal{A}_{L/K,\,i}$ spans $\diff^{\,p^i}_KL$ as an $L$ subspace.
\end{theorem} 

\begin{proof}
Let $\mathcal{A}_{L/K,\,i}$ be denoted $\mathcal{A}_i$ in this proof.  Suppose $L/K$ is a modular extension.  By definition there exist $\{x_1,\ldots x_n\}\subset L$ such that $\displaystyle L\cong\bigotimes_{i=1}^n K(x_i)$.  Then 
\begin{equation}\label{eq:tensordiffop}
\displaystyle \diff_KL\cong \bigotimes_{i=1}^n \diff_KK(x_i).
\end{equation}  Now, if $l_i=\exp\left[x_i:K\right]$, then $\diff_KK(x_i)$ is generated as a $K(x_i)$-algebra by $\displaystyle \left\{\left(\frac{d}{dx_i}\right)^{[p^k]}\right\}_{0\le k<l_i}$.  Hence, by the isomorphism in Equation \ref{eq:tensordiffop}, any $D\in\diff_KL$ is a linear combination over $L$ of differential operators of the form $$D_{a_1,\ldots,a_n}\coloneqq\left(\frac{d}{dx_1}\right)^{[a_1]}\cdots\left(\frac{d}{dx_n}\right)^{[a_n]}$$ where $0\le a_j<l_j$.  $D_{a_1,\ldots,a_n}\in \mathcal{A}_i$ if and only if $a_1+\cdots+a_n\le p^i$, so for some integer $j$ $D_{a_1,\ldots,a_n}$ is an element of $\mathcal{A}_j$.  Thus every $D\in \diff^{\,p^i}_KL$ is the $L$-linear combination of elements in $\mathcal{A}_i$.  Hence $\mathcal{A}_i$ spans $\diff^{\,p^i}_KL$.\\
\indent
Now suppose $L/K$ is not a modular extension.  We will find elements $a,\,b\in L$ with $a\in L^{p^i}$ and $b\in L^{p^j}$ and a differential operator $D\in \diff^{\,p^l}_KL$, $i,\,j\le l$ that satisfy the following properties: $D$ is not the sum of products of differential operators of lower order, $D(a)\in L^{p^i}$, and $D(b)\notin L^{p^j}$.  This operator $D$ will then not be spanned by $\mathcal{A}_i$.

 Let $\{x_1,\ldots, x_n\}$ be a Pickert generating sequence for $L/K$ with $e_1\ge e_2\ge\cdots\ge e_n$ the intrinsic non-increasing sequence of exponents.  Since $L/K$ is not modular, by Proposition \ref{prop:pickertmodular} there exists an integer $i$ satisfying 
 $$\exp\left[x_i:K(x_1,\ldots,x_{i-1}\right]\ne\exp\left[x_i:K\right],$$ 
 and $x_i$ has a structure equation with coefficients which are not all in $L^{p^{e_i}}$.

Let $z\coloneqq x_i$ be the first such element in the Pickert Generating sequence to exhibit this property.  Set $q=p^{e_i}$ and let $z^q=f(x_1^q,\ldots,x_{i-1}^q)$ be the structure equation of $z$.  Since
$z^q\notin K$, $f$ is neither constant nor does the degree of every $x_j^q$ in the polynomial exceed $p^{e/e_i}$.\\
\indent The polynomial $f$ must have at least two terms with coefficients which are not $q$th powers in $L$.  That is, suppose $\alpha x_1^{qa_1}\ldots x_{i-1}^{qa_{i-1}}$ is the only summand of $f$ with $\alpha\in K\setminus L^q$.  Then solving the structure equation of $z$ for $\alpha$, we get
$$\alpha=\frac{z^q-f(x_1^q,\ldots x_{i-1}^q)+\alpha x_1^{qa_1}\ldots x_{i-1}^{qa_{i-1}}}{x_1^{qa_1}\ldots x_{i-1}^{qa_{i-1}}}.$$
The right-hand side is a $q$th power in $L$, which contradicts $\alpha\notin L^q$.  
 
\indent Thus, at least two of the coefficients of terms of $f$ are not in $L^q$.  Let $\mathcal{C}$ denote the set of all such monomials of $f$.  $\mathcal{C}$ has at least two elements, so $f$ must have $K$-independent summands $ax_1^{qa_1}\cdots x_{i-1}^{qa_{i-1}}$ and $bx_1^{qb_1}\cdots x_{i-1}^{qb_{i-1}}$ where $a$ and $b$ are not in $L^q$ and $a_j,\,b_j<p^{e_j-e_i}$ for all $1\le j\le i-1$.  At least one of the $x_j$ must have different degrees in these summands or else they are not linearly independent.  Suppose without loss of generality that $x_1$ is the element with $a_1\ne b_1$ and that $b_1$ is the largest such exponent of $x_1^q$ for such an element of $\mathcal{C}$.  Define $Q$ as the largest power of $p$ which is $\le qb_1$.  Write
$$f(x_1^q\ldots,x_{i-1}^q)=\sum_{0\le j_k<p^{e_k-e_i}}\alpha_{j_1,\ldots,j_{i-1}}x_1^{qj_1}\cdots x_{i-1}^{qj_{i-1}}$$

Now, $\displaystyle \left(\frac{d}{dx_1}\right)^{[Q]}$ 
is a differential operator of
$\displaystyle K(x_1,\ldots, x_{i-1})=\bigotimes_{j=1}^{i-1}K(x_j)$ over 
$K$ of order $Q$.  By Lemma \ref{lem:diffopextension}, we can extend $\displaystyle \left(\frac{d}{dx_1}\right)^{[Q]}$ to a differential operator $D\in\diff_KK(x_1,\ldots, x_i)$ of order $Q$.  By induction on $i$ we can extend $D$ in this manner to a differential operator on $L$.  Call $\widetilde{D}$ the extension of $\displaystyle\left(\frac{d}{dx_1}\right)^{[Q]}$ to $L$.  Gerstenhaber \cite{GerstFundForm} calls $\widetilde{D}$ the \emph{normal extension of $\displaystyle\left(\frac{d}{dx_1}\right)^{[Q]}$}, which is a differential operator of order $Q$ by Lemma \ref{lem:diffopextension}.

So, 
\begin{eqnarray*}
\widetilde{D}(z^q)&=&D\left(f(x_1^q,\ldots,x_{i-1}^q)\right)\\&=&
\left(\frac{d}{dx_1}\right)^{[Q]}\left(\sum_{0\le j_k<p^{e_k-e_i}}\alpha_{j_1,\ldots,j_{i-1}}x_1^{qj_1}\cdots x_{i-1}^{qj_{i-1}}\right)\\
&=& \sum_{0\le j_k<p^{e_k-e_i}} \alpha_{j_1,\ldots,j_{i-1}}\binom{qj_1}{Q}(x_1)^{qj_1-Q}(x_2^q)^{j_2}\cdots(x_{i-1}^q)^{j_{i-1}}
\end{eqnarray*}

This last sum is not zero because, by the choice of $x_1$, $f$ has at least one nonzero term with the degree of $x_1^q$ greater than $0$.  Additionally, by the choice of $Q$, $\displaystyle \binom{qj_1}{Q}\ne 0$ for some $j_1$ \cite[p.577]{Alperin}.  Likewise, the sum is not a $q$th power in $L$, because at least one of the coefficients $\alpha_{j_1,\ldots,j_{i-1}}$ is neither a $q$th power in $L$ nor the coefficient of a term which vanishes by $\widetilde{D}$.  Therefore, $\widetilde{D}(z^q)\notin L^q$ but $\widetilde{D}(x_1^Q)=1\in L^Q$.  Since $\displaystyle \left(\frac{d}{dx_1}\right)^{[Q]}$ is not the product of differential operators of lower order in $\diff_KK(x_1,\ldots, x_{i-1})$, $\widetilde{D}$ is not the product of differential operators of lower order either.  Therefore, $\widetilde{D}$ is not in the $L$-span of $\mathcal{A}_i$ for any $i$ and the theorem is proven.
\end{proof}



An immediate corollary of the theorem is

\begin{corollary}\label{cor:biggestmodularextension}
Let $L/K$ be as in the theorem.  Let $\mathcal{D}$ be the $L$-subalgebra of $\diff_KL$ generated by the $\mathcal{A}_i$.  Then $\mathcal{D}$ is the largest subalgebra of $\diff_KL$ such that $L/L^{\mathcal{D}}$ is a modular extension.
\end{corollary}

\bibliographystyle{amsplain}
\bibliography{Dissertation_Bibliography}
\end{document}